\documentclass[11pt]{amsart}
\usepackage[margin=1.15in]{geometry}
\usepackage{amscd,amssymb, amsmath, wasysym}
\usepackage{graphicx}
\usepackage{amsfonts}
\usepackage{mathrsfs}    
\usepackage{amsmath}    
\usepackage{amsthm}     
\usepackage{amscd}      
\usepackage{amssymb}    
\usepackage{eucal}      
\usepackage{latexsym}   
\usepackage{verbatim}   
\usepackage[all]{xy}     
\usepackage[dvipsnames]{xcolor}

\usepackage{mathtools}

\usepackage[utf8]{inputenc}

\usepackage{bookmark}

 \usepackage{hyperref}
 \hypersetup{
     colorlinks=true,
     linkcolor=NavyBlue,
     filecolor=NavyBlue,
     citecolor = TealBlue,
     urlcolor=cyan,
  }

\newcounter{thmcounter}

\numberwithin{equation}{section}
\numberwithin{thmcounter}{section}

\newtheorem{theorem}[thmcounter]{Theorem}
\newtheorem{proposition}[thmcounter]{Proposition}
\newtheorem{lemma}[thmcounter]{Lemma}

\newtheorem{question}[thmcounter]{Question}
\newtheorem{conjecture}[thmcounter]{Conjecture}
\newtheorem{problem}[thmcounter]{Problem}

\theoremstyle{definition}

\newtheorem{example}[thmcounter]{Example}

\theoremstyle{remark}
\newtheorem{remark}[thmcounter]{Remark}

\newtheoremstyle{claim}{9pt}{3pt}{}{\parindent}{\bf}{.}{1em}{}

\theoremstyle{claim}

\newenvironment{namelist}[1]{%
\begin{list}{}
{
\settowidth{\labelwidth}{#1}%
\setlength{\labelsep}{0.3em}%
\setlength{\leftmargin}{\labelwidth}%
\addtolength{\leftmargin}{\labelsep}}}{%
\end{list}}

\newcommand{\nP}{\mathbf{P}}                     

\newcommand{\sF}{\mathscr{F}}

\newcommand{\sO}{\mathscr{O}}                    

\DeclareMathOperator{\Char}{char}                
\DeclareMathOperator{\coker}{coker}              

\DeclareMathOperator{\ev}{ev}					 

\DeclareMathOperator{\id}{id}                    

\DeclareMathOperator{\Supp}{Supp}                

\DeclareMathOperator{\rank}{rank}                

\newcounter{rkcounter}             
\begin{document}

\title[Asymptotic nonvanishing of syzygies of algebraic varieties]{Asymptotic nonvanishing of syzygies \\of algebraic varieties}

\author{Jinhyung Park}
\address{Department of Mathematical Sciences, KAIST, 291 Daehak-ro, Yuseong-gu, Daejeon 34141, Republic of Korea}
\email{parkjh13@kaist.ac.kr}

\date{\today}
 
\thanks{J. Park was partially supported by the National Research Foundation (NRF) funded by the Korea government (MSIT) (NRF-2019R1A6A1A10073887 and NRF-2022M3C1C8094326).}

\begin{abstract} 
We establish precise nonvanishing results for asymptotic syzygies of smooth projective varieties. This refines Ein--Lazarsfeld's asymptotic nonvanishing theorem. Combining with the author's previous asymptotic vanishing result, we completely determine the asymptotic shapes of the minimal free resolutions of the graded section modules of a line bundle on a smooth projective variety as the positivity of the embedding line bundle grows. 
\end{abstract}

\maketitle


\section{Introduction}
The purpose of this paper is to address the main theme of \cite{EL2} and \cite{Park}:  \emph{the asymptotic behavior of syzygies of algebraic varieties is surprisingly uniform}. After the pioneering work of Green \cite{Green}, there has been a considerable amount of research to understand syzygies of algebraic varieties. It is an interesting problem to describe the overall asymptotic behavior of syzygies of graded section modules of a line bundle on a smooth projective variety as the positivity of the embedding line bundle grows (see \cite[Problem 5.13]{Green} and \cite[Problem 4.4]{EL1}). The influential paper \cite{EL2} of Ein--Lazarsfeld opens the door to asymptotic syzygies of algebraic varieties, and the asymptotic nonvanishing theorem was proved there. In the present paper, we provides a new approach to nonvanishing of asymptotic syzygies, and together with the author's asymptotic vanishing theorem \cite{Park}, we exhibit the uniform behavior of all asymptotic syzygies.

\medskip

Throughout the paper, we work over an algebraically closed field $\mathbf{k}$ of arbitrary characteristic. Let $X$ be a smooth projective variety of dimension $n$, and $B$ be a line bundle on $X$. For an integer $d \geq 1$, set
$$
L_d:=\sO_X(dA+P)~~\text{ and }~~r_d:=h^0(X, L_d)-1,
$$
where $A$ is an ample divisor and $P$ is an arbitrary divisor on $X$. We assume that $d$ is sufficiently large so that $L_d$ is very ample and $r_d = \Theta(d^n)$. Here, for a nonnegative function $f(d)$ defined for positive integers $d$, we define:
$$
\begin{array}{rcl}
f(d) \geq \Theta(d^k) & \Longleftrightarrow & \begin{array}{l} \text{there is a constant $C_1 > 0$ such that} \\ \text{$f(d) \geq C_1 d^k$ for any sufficiently large positive integer $d$};\end{array}\\[12pt]
f(d) \leq \Theta(d^k) & \Longleftrightarrow & \begin{array}{l} \text{there is a constant $C_2>0$ such that} \\ \text{$f(d) \leq C_2 d^k$ for any sufficiently large positive integer $d$};\end{array}\\[12pt]
f(d) = \Theta(d^k) & \Longleftrightarrow & \begin{array}{l} \text{there are constants $C_1, C_2 > 0$ such that 
} \\ \text{$C_1 d^k \leq f(d) \leq C_2 d^k$ for any sufficiently large positive integer $d$.}\end{array}
\end{array}
$$
For simplicity, we write $f(d) = \Theta(1)$ if $f(d)$ is a constant including $0$ for any sufficiently large positive integer $d$.
Let $S_d:=\bigoplus_{m \geq 0} S^mH^0(X, L_d)$. By Hilbert syzygy theorem, the finitely generated graded section $S_d$-module
$$
R_d=R(X, B; L_d):=\bigoplus_{m \in \mathbf{Z}} H^0(X, B \otimes L_d^m)
$$
admits a minimal free resolution
$$
 \xymatrix{
 0 & R_d \ar[l]& E_0 \ar[l]  & E_1  \ar[l] & \ar[l] \cdots & \ar[l] E_r \ar[l] & \ar[l] 0, 
 }
$$
where
$$
E_p = \bigoplus_{q} K_{p,q}(X, B; L_d)  \otimes_{\mathbf{k}} S_d(-p-q).
$$
Here the \emph{Koszul cohomology group} $K_{p,q}(X, B; L_d)$ can be regarded as the space of $p$-th syzygies of weight $q$.  When $X \subseteq \nP H^0(X, L_d) = \nP^{r_d}$ is projectively normal, the minimal free resolution of $R_d$ with $B=\sO_X$ contains all information about the defining equations of $X$ in $\nP^{r_d}$ and their syzygies. Based on the experience of the case of curves, it was widely believed that the minimal free resolutions of $R_d$ become simpler as $d$ increases. However, Ein--Lazarsfeld \cite{EL2} showed that this had been misleading, and they instead proposed that there would be a uniform asymptotic vanishing and nonvanishing behavior of $K_{p,q}(X, B; L_d)$ when $d$ is sufficiently large.

\medskip

It is elementary to see that
$$
K_{p,q}(X, B; L_d) = 0~~\text{ for $q \geq n+2$}.
$$
The cases $q=0$ and $q=n+1$ are well understood due to Green, Schreyer, Ottaviani--Paoletti, and Ein--Lazarsfeld: \cite[Proposition 5.1 and Corollary 5.2]{EL2} state that
$$
\begin{array}{rcl}
K_{p,0}(X, B; L_d) \neq 0 & \Longleftrightarrow & 0 \leq p \leq h^0(B) -1;\\[3pt]
K_{p,n+1}(X, B; L_d) \neq 0 & \Longleftrightarrow & r_d - n - h^0(X, \omega_X \otimes B^{-1}) + 1 \leq p \leq r_d - n.
\end{array}
$$
For $1 \leq q \leq n$, let $c_q(d)$ be the number such that
$$
K_{c_q(d), q}(X, B; L_d) \neq 0~~\text{ and }~~K_{p, q}(X, B; L_d) = 0~~\text{ for $0 \leq p \leq c_q(d)-1$},
$$
and $c_q'(d)$ be the number such that
$$
K_{r_d-c_q'(d), q}(X, B; L_d) \neq 0~~\text{ and }~~K_{p,q}(X, B; L_d) = 0~~\text{ for $r_d-c_q'(d)+1 \leq p \leq r_d$}.
$$
After interesting nonvanishing results of Ottaviani--Paoletti \cite{OP} and Eisenbud--Green--Hulek--Popescu \cite{EGHP}, Ein--Lazarsfeld proved the asymptotic nonvanishing theorem (\cite[Theorem 4.1]{EL2}): For each $1 \leq q \leq n$, if $d$ is sufficiently large, then
$$
K_{p,q}(X, B; L_d) \neq 0~~\text{ for $\Theta(d^{q-1}) \leq p \leq r_d - \Theta(d^{n-1})$}.
$$
In particular, $c_q(d) \geq \Theta(d^{q-1})$ and $c_q'(d) \leq \Theta(d^{n-1})$. In  \cite[Conjecture 7.1]{EL2}, Ein--Lazarsfeld conjectured that 
$$
K_{p,q}(X, B; L_d) = 0~~\text{ for $0 \leq p \leq \Theta(d^{q-1})$},
$$
and this was confirmed by the present author \cite[Theorem 1.1]{Park} using Raicu's result in the appendix of \cite{Raicu}. In particular, $c_q(d) = \Theta(d^{q-1})$.

\medskip

Despite the aforementioned results on asymptotic syzygies of algebraic varieties, at least two problems remain. First, it is unclear whether vanishing and nonvanishing of $K_{p,q}(X, B; L_d)$ can alternate in a few steps after $c_q(d)$ or before $r_d-c_q'(d)$. Second, the previous results do not say anything about  $K_{p,q}(X, B; L_d)$ for $r_d - \Theta(d^{n-1}) \leq p \leq r_d$. In this paper, we completely resolve these two issues: We show that vanishing and nonvanishing of $K_{p,q}(X, B; L_d)$ can alternate only at $c_q(d)$ and $r_d-c_q'(d)$, and we give estimations for $c_q(d)$ and $c_q'(d)$. Consequently, we could determine the precise vanishing and nonvanishing range of $p$ for $K_{p,q}(X, B; L_d)$. It is worth noting that $c_q'(d)$ heavily depends on $H^{q-1}(X, B)$ while $c_q(d)$ depends only on $d$ asymptotically. 

\begin{theorem}\label{thm:main}
Let $X$ be a smooth projective variety of dimension $n \geq 1$, and $B$ be a line bundle on $X$. For an integer $d \geq 1$, set
$$
L_d:=\sO_X(dA+P)~~\text{ and }~~r_d:=h^0(X, L_d)-1,
$$
where $A$ is an ample divisor and $P$ is an arbitrary divisor on $X$. Fix an index $1 \leq q \leq n$. Then there exist functions $c_q(d)$ and $c_q'(d)$ with
\begin{equation}\label{eq:main1}
c_q(d) = \Theta(d^{q-1})~~\text{ and }~~c_q'(d) = \begin{cases} \Theta(d^{n-q}) & \text{if $H^{q-1}(X, B) =0$ or $q=1$} \\  q-1 & \text{if $H^{q-1}(X, B) \neq 0$ and $q \geq 2$} \end{cases}
\end{equation}
such that if $d$ is sufficiently large, then
\begin{equation}\label{eq:main2}
K_{p,q}(X, B; L_d) \neq 0~~\Longleftrightarrow~~c_q(d) \leq p \leq r_d-c_q'(d).
\end{equation}
\end{theorem}

To prove Theorem \ref{thm:main}, we do not use Ein--Lazarsfeld's asymptotic nonvanishing theorem \cite[Theroem 4.1]{EL2}, but we adopt their strategy in \cite{EL2}. Let $H$ be a suitably positive very ample line bundle on $X$, and choose a general member $\overline{X} \in |H|$. Put
$$
\overline{L}_d:=L_d|_{\overline{X}},~~\overline{B}:=B|_{\overline{X}},~~\overline{H}:=H|_{\overline{X}}, ~~V_d':=H^0(X, L_d \otimes H^{-1}).
$$
We have noncanonical splitting\\[-20pt]

\begin{footnotesize}
$$
K_{p,q}(X, \overline{B}; L_d) = \bigoplus_{j=0}^p \wedge^j V_d' \otimes K_{p-j, q}(\overline{X}, \overline{B}; \overline{L}_d)~~\text{ and }~~
K_{p,q}(X, \overline{B} \otimes \overline{H}; L_d) =   \bigoplus_{j=0}^p \wedge^j V_d' \otimes K_{p-j, q}(\overline{X}, \overline{B} \otimes \overline{H}; \overline{L}_d).
$$
\end{footnotesize}

\noindent There are natural maps
$$
\theta_{p,q} \colon K_{p+1, q-1}(X, \overline{B} \otimes \overline{H}; L_d) \longrightarrow K_{p,q}(X, B; L)~~\text {and }~~\theta_{p,q}' \colon K_{p,q}(X, \overline{B}; L_d) \longrightarrow K_{p,q}(X, B; L_d).
$$
When $q=1$ or $2$, the map $\theta_{p,q}$ should be modified (see Section \ref{sec:lifting}). In \cite[Sections 3 and 4]{EL2}, the secant constructions were introduced to carry nonzero syzygies by highly secant planes. This essentially shows that  the map $\theta_{p,q}$ is nonzero for $\Theta(d^{q-1}) \leq p \leq r_d - \Theta(d^{n-1})$. Instead of utilizing the secant constructions, in this paper, we apply the asymptotic vanishing theorem \cite[Theorem 1.1]{Park} to get the estimations of $c_q(d)$ and $c_q'(d)$ (Proposition \ref{prop:Theta}) and to see that the maps $\theta_{c_q(d), q}$ and $\theta_{r_d - c_q'(d), q}'$ are nonzero (Proposition \ref{prop:theta}). The latter statement means that there are syzygies $\alpha$ and $\beta$ in 
\begin{equation}\label{eq:syzinH}\tag{$\star$}
K_{c_q(d)+1-j_0, q-1}(\overline{X}, \overline{B} \otimes \overline{H}; \overline{L}_d)~~\text{ and }~~K_{r_d-c_q'(d)-j_0', q}(\overline{X}, \overline{B}; \overline{L}_d)
\end{equation}
for some $0 \leq j_0 \leq c_q(d)+1$ and $\dim V_d' - c_q'(d) \leq j_0' \leq \dim V_d'$ that are lifted to syzygies in $K_{c_q(d), q}(X, B; L_d)$ and $K_{r_d - c_q'(d), q}(X, B; L_d)$ via the maps $\theta_{c_q(d), q}$ and $\theta_{r_d - c_q'(d), q}'$, respectively. Since (\ref{eq:syzinH}) survives in\\[-22pt]

\begin{small}
$$
\text{$K_{p,q}(X, \overline{B} \otimes \overline{H}; L_d)$ for $c_q(d) \leq p \leq r_d - \Theta(d^{n-1})$ and $K_{p,q}(X, \overline{B}; L_d)$ for $\Theta(d^{n-1}) \leq p \leq r_d - c_q'(d)$},
$$
\end{small}

\noindent we can argue that the syzygies $\alpha$ and $\beta$ in  (\ref{eq:syzinH}) are also lifted to syzygies in $K_{p,q}(X, B; L_d)$ for $c_q(d) \leq p \leq r_d - c_q'(d)$ via the maps $\theta_{p, q}$ and $\theta_{p, q}'$, respectively (Theorem \ref{thm:lifting}).

\medskip

The paper is organized as follows. After recalling preliminary results on syzygies of algebraic varieties in Section \ref{sec:prelim}, we show how to lift syzygies from hypersurfaces (Theorem \ref{thm:lifting}) in Section \ref{sec:lifting}. Section \ref{sec:proof} is devoted to the proof of Theorem \ref{thm:main}. Finally, in Section \ref{sec:openprob}, we present complementary results and open problems on asymptotic syzygies of algebraic varieties.

\subsection*{Acknowledgements}
The author is very grateful to Lawrence Ein, Yeongrak Kim, and Wenbo Niu for inspiring discussions.

\section{Preliminaries}\label{sec:prelim}

In this section, we collect relevant basic facts on Koszul cohomology and Castelnuovo--Mumford regularity.

\subsection{Koszul Cohomology}
Let $V$ be an $r$-dimensional vector space over an algebraically closed field $\mathbf{k}$, and  $S:= \bigoplus_{m \geq 0} S^m V$. Consider a finitely generated graded $S$-module $M$. The \emph{Koszul cohomology group} $K_{p,q}(M, V)$ is the cohomology of the Koszul-type complex
$$
\wedge^{p+1} V \otimes M_{q-1} \xrightarrow{~\delta~} \wedge^p V \otimes M_q \xrightarrow{~\delta~} \wedge^{p-1} V \otimes M_{q+1},
$$
where the Koszul differential $\delta$ is given by
$$
\delta(s_1 \wedge \cdots \wedge s_p \otimes t) \longmapsto \sum_{i=1}^p (-1)^i s_1 \wedge \cdots \wedge \widehat{s}_i \wedge \cdots \wedge s_p \otimes s_i t.
$$
Then $M$ has a minimal free resolution
$$
 \xymatrix{
 0 & M \ar[l]& E_0 \ar[l]  & E_1  \ar[l] & \ar[l] \cdots & \ar[l] E_r \ar[l]  & \ar[l] 0, 
 }
$$
where
$$
E_p = \bigoplus_{q} K_{p,q}(M, V)  \otimes_{\mathbf{k}} S(-p-q).
$$
We may regard $K_{p,q}(M,V)$ as the vector space of $p$-th syzygies of weight $q$. Let 
$$
0 \longrightarrow M' \longrightarrow M \longrightarrow M'' \longrightarrow 0
$$
be a short exact sequence of finitely generated graded $S$-modules. By \cite[Corollary (1.d.4)]{Green} (see also \cite[Lemma 1.24]{AN}), this induces a long exact sequence
$$
\cdots \longrightarrow K_{p+1, q-1}(M, V) \longrightarrow K_{p+1, q-1}(M'', V) \longrightarrow K_{p,q}(M', V) \longrightarrow K_{p,q}(M, V) \longrightarrow \cdots.
$$

\medskip

Consider an injective map
$$
\iota' \colon \wedge^{p+1} V \longrightarrow V \otimes \wedge^p V,~~s_1 \wedge \cdots \wedge s_{p+1} \longmapsto \sum_{i=1}^{p+1} (-1)^i s_i \otimes s_1 \wedge \cdots \wedge \widehat{s}_i \wedge \cdots \wedge s_{p+1}.
$$
It is straightforward to check that the following diagram commutes:
$$
\xymatrixcolsep{0.69in}
\xymatrix{
\wedge^{p+1} V \otimes M_q \ar[r]^-{\delta} \ar[d]_-{\iota' \otimes \id_{M_q}} & \wedge^p V \otimes M_{q+1} \ar[d]^-{\iota' \otimes  \id_{M_{q+1}}} \\
V \otimes \wedge^p V \otimes M_q \ar[r]_-{-\id_{V} \otimes \delta} & V \otimes \wedge^{p-1} V \otimes M_{q+1}.
}
$$
Then the map $\iota'$ induces a map
$$
\iota \colon K_{p+1,q}(M, V) \longrightarrow V \otimes K_{p,q}(M, V).
$$
This map is the glueing of the evaluation maps $\ev_x \colon K_{p+1,q}(M, V) \longrightarrow K_{p,q}(M, V)$ for $x \in V^{\vee}$ in \cite[Subsection 2.2.1]{AN}.

\medskip

We now turn to the geometric setting. Let $X$ be a projective variety, $B$ be a coherent sheaf on $X$, and $L$ be a very ample line bundle on $X$. Put $V:=H^0(X, L)$ and $S:=\bigoplus_{m \geq 0} S^m V$. Then the section module
$$
R(X, B; L):=\bigoplus_{m \in \mathbf{Z}} H^0(X, B \otimes L^m),
$$
is finitely generated graded $S$-module. We define the \emph{Koszul cohomology group} as
$$
K_{p,q}(X, B; L):=K_{p,q}(R(X,B;L), V).
$$
In this paper, $L$ is always assumed to be sufficiently positive, so we  have 
$$
H^0(X, B \otimes L^{-m})=0~~\text{ for $m>0$}. 
$$
Then $K_{p,q}(X, B; L) = 0$ for $q < 0$. It is clear that if $K_{p_0,q}(X, B; L)=0$ for $0 \leq q \leq q_0$, then $K_{p, q}(X, B; L ) = 0$ for $p \geq p_0$ and $0 \leq q \leq q_0$. Let $M_L$ be the kernel bundle of the evaluation map $\ev \colon H^0(X, L)\otimes \sO_{X} \rightarrow  L$. We have a short exact sequence
$$
0 \longrightarrow M_L \longrightarrow H^0(X, L) \otimes \sO_X \longrightarrow L \longrightarrow 0.
$$
We frequently use the following well-known facts.

\begin{proposition}[{cf. \cite[Proposition 3.2, Corollary 3.3, Remark 3.4]{EL2}, \cite[Proposition 2.1]{Park}}]\label{prop:koszulcoh}
Assume that 
$$
H^i(X, B \otimes L^m)=0~~\text{ for $i >0$ and $m >0$.}
$$ 
For any $p \geq 0$, the following hold:\\[3pt]
$(1)$ If $q \geq 2$, then 
$$
\begin{array}{rcl}
K_{p,q}(X, B; L)&=& H^1(X, \wedge^{p+1} M_L \otimes B \otimes L^{q-1})\\
&=&  H^2(X, \wedge^{p+2} M_L \otimes B \otimes L^{q-2})\\
&  & ~\text{ }~\text{ }~\text{ } ~\text{ }~\text{ }~\text{ }~\text{ }~ \vdots \\
&=& H^{q-1}(X, \wedge^{p+q-1}M_L \otimes B \otimes L).
\end{array}
$$
Consequently, $K_{p, q}(X, B; L_d) = 0$ for $p \geq r_d-q$, and $K_{p,q}(X, B; L_d) = 0$ for $q \geq \dim X+2$.\\[3pt]
$(2)$ If $q \geq 2$ and $H^{q-1}(X, B) = H^q(X, B) = 0$, then $K_{p,q}(X, B; L) = H^q(X, \wedge^{p+q} M_L \otimes B)$.\\[3pt]
$(3)$ If $q=1$, then 
$$
K_{p,1}(X, B; L) = \coker\big( \wedge^{p+1} H^0(X, L) \otimes H^0(X, B) \longrightarrow H^0(X,  \wedge^p M_L \otimes B \otimes L) \big).
$$
If $H^1(X, B) =0$, then $K_{p,1}(X, B; L) = H^1(X, \wedge^{p+1} M_L \otimes B)$.\\[3pt]
$(4)$ If $q=0$ and $H^0(X, B \otimes L^{-m})=0$ for $m > 0$, then $K_{p,0}(X, B;L ) = H^0(X, \wedge^{p} M_L \otimes B)$.
\end{proposition}

\begin{proof}
We have a short exact sequence
\begin{equation}\label{eq:wedgeM_L}
0 \longrightarrow \wedge^{p+1} M_L \longrightarrow \wedge^{p+1} H^0(X, L) \otimes \sO_X \longrightarrow  \wedge^{p} M_L \otimes L \longrightarrow  0.
\end{equation}
Using the Koszul-type complex and chasing through the diagram, we see that
$$
K_{p,q}(X, B; L) = \coker\big(\wedge^{p+1} H^0(X, L) \otimes H^0(X, B \otimes L^{q-1}) \longrightarrow H^0(X, \wedge^p M_L \otimes B \otimes L^q) \big).
$$
Then the proposition easily follows. See \cite[Section 2.1]{AN}, \cite[Section 1]{EL1}, \cite[Section 3]{EL2}. 
\end{proof}

\begin{proposition}[{cf. \cite[Proposition 2.4]{Agostini}, \cite[Proposition 3.5]{EL2}}]\label{prop:duality}
Put $n:=\dim X$. Assume that $X$ is smooth, $B$ is a line bundle, and 
$$
H^i(X, B \otimes L^m)=0~~\text{ for $i >0$ and $m >0$ or $i< n$ and $m<0$.}
$$ 
For any $p \geq 0$, the following hold:\\[3pt]
$(1)$ If $q=n+1$, then $K_{p, n+1}(X, B; L) = K_{r-p-n, 0}(X, \omega_X \otimes B^{-1}; L)^{\vee}$.\\[3pt]
$(2)$ If $q=n$, then there is an exact sequence
$$
\wedge^{p+n} H^0(X, L) \otimes H^{n-1}(X, B) \longrightarrow K_{p,n}(X, B; L) \longrightarrow K_{r-p-n, 1}(X, \omega_X \otimes B^{-1}; L_d)^{\vee} \longrightarrow 0.
$$
If $H^{n-1}(X, B) = 0$, then $K_{p,n}(X, B; L) = K_{r-p-n, 1}(X, \omega_X \otimes B^{-1}; L)^{\vee}$.\\[3pt]
$(3)$ If $2 \leq q \leq n-1$, then there is an exact sequence
$$
\begin{array}{l}
\wedge^{p+q} H^0(X, L) \otimes H^{q-1}(X, B) \longrightarrow K_{p,q}(X, B; L) \\[3pt]
\text{ }~\text{ }~\text{ }\longrightarrow K_{r-p-n, n+1-q}(X, \omega_X \otimes B^{-1}; L_d)^{\vee}  \longrightarrow \wedge^{p+q} H^0(X, L) \otimes H^q(X, B).
\end{array}
$$
If $H^{q-1}(X, B) = H^q(X, B)=0$, then $K_{p,q}(X, B; L) = K_{r-p-n, n+1-q}(X, \omega_X \otimes B^{-1}; L_d)^{\vee}$.\\[3pt]
$(4)$ If $q=1$, then there is an exact sequence
$$
0 \longrightarrow K_{p,1}(X, B; L) \longrightarrow K_{r-p-n, n}(X, \omega_X \otimes B^{-1}; L_d)^{\vee}  \longrightarrow \wedge^{p+1} H^0(X, L) \otimes H^1(X, B)
$$
If $H^1(X, B)=0$, then $K_{p,1}(X, B; L) = K_{r-p-n, n}(X, \omega_X \otimes B^{-1}; L_d)^{\vee}$.\\[3pt]
$(5)$ If $q=0$, then $K_{p,0}(X, B; L) = K_{r-p-n, n+1}(X, \omega_X \otimes B^{-1}; L)^{\vee}$.
\end{proposition}

\begin{proof}
Let $V:=H^0(X, L)$.
From (\ref{eq:wedgeM_L}), we get an exact sequence\\[-25pt]

\begin{footnotesize}
$$
\wedge^{p+q}V \otimes H^{q-1}(X, B) \longrightarrow H^{q-1}(X, \wedge^{p+q-1} M_L \otimes B \otimes L) \longrightarrow H^q(X, \wedge^{p+q} M_L \otimes B) \longrightarrow \wedge^{p+q} V \otimes H^q(X, B).
$$
\end{footnotesize}

\noindent By Proposition \ref{prop:koszulcoh}, $H^{q-1}(X, \wedge^{p+q-1} M_L \otimes B \otimes L) = K_{p,q}(X, B; L)$ when $q \geq 2$. By Serre duality,
$$
H^q(X, \wedge^{p+q} M_L \otimes B)  = H^{n-q}(X, \wedge^{p+q} M_L^{\vee} \otimes \omega_X \otimes B^{-1})^{\vee}.
$$
Since $\rank M_L = r$ and $\det M_L = L^{-1}$, it follows that $\wedge^{p+q} M_L^{\vee} = \wedge^{r-p-q} M_L \otimes L$. Thus 
$$
H^q(X, \wedge^{p+q} M_L \otimes B) = H^{n-q}(X, \wedge^{r-p-q} M_L \otimes \omega_X \otimes B^{-1} \otimes L)^{\vee}.
$$
Using Proposition \ref{prop:koszulcoh}, the assertions easily follow. See \cite[Section 2.3]{AN}, \cite[Section 3]{EL2}.
\end{proof}

In the situation of Proposition \ref{prop:duality}, if we further assume $H^i(X, B)=0$ for $1 \leq i \leq n-1$, i.e., $R(X, B; L)$ and $R(X, \omega_X \otimes B^{-1}; L)$ are Cohen--Macaulay, then 
$$
K_{p,q}(X, B; L) = K_{r-p-n, n+1-q}(X, \omega_X \otimes B^{-1}; L)^{\vee}.
$$

\begin{lemma}\label{lem:technical1}
If $H^q(X, M_{L} \otimes \wedge^{p} M_{L} \otimes B)=0$ and  $H^q(X, B)=0$, then $H^q(X, \wedge^{p+1} M_{L} \otimes B)=0$.
\end{lemma}

\begin{proof}
 Let $V:=H^0(X, L)$. Consider the commutative diagram with exact rows
$$
\xymatrix{
0 \ar[r]& M_{L} \otimes \wedge^p M_{L} \ar[r] \ar@{->>}[d] & V \otimes  \wedge^p M_{L} \ar@{->>}[d] \ar[r] & \wedge^p M_{L} \otimes L \ar[r] \ar@{=}[d] & 0\\
0 \ar[r] & \wedge^{p+1} M_{L} \ar[r] &\wedge^{p+1} V  \otimes \sO_{X} \ar[r] & \wedge^p M_{L} \otimes L \ar[r] & 0,
}
$$
which gives rise to the following commutative diagram with exact rows
$$
\xymatrixcolsep{0.23in}
\xymatrix{
 H^{q-1}(X, \wedge^p M_{L} \otimes B \otimes L) \ar[r] \ar@{=}[d] & H^q(X, M_{L} \otimes \wedge^p M_{L} \otimes B) \ar[d] \ar[r] & V \otimes H^q(X, \wedge^p M_{L} \otimes B) \ar[d] \\
 H^{q-1}(X, \wedge^p M_{L} \otimes B \otimes L) \ar[r]  & H^q(X, \wedge^{p+1} M_{L} \otimes B) \ar[r] & \wedge^{p+1} V \otimes H^q(X, B).
}
$$
Since $H^q(X, B)=0$, the middle vertical map is surjective. Thus the lemma follows.
\end{proof}

\subsection{Castelnuovo--Mumford Regularity}
Let $X$ be a projective variety, and $L$ be a very ample line bundle on $X$. A coherent sheaf $\sF$ on $X$ is said to be \emph{$m$-regular} with respect to $L$ if 
$$
H^i(X, \sF \otimes \sO_X(m-i))=0~~\text{ for $i>0$}.
$$
By Mumford's theorem (\cite[Theorem 1.8.5]{positivity}), if $\sF$ is $m$-regular with respect to $L$, then $\sF \otimes L^{m+\ell}$ is globally generated, the multiplication map
$$
H^0(X, \sF \otimes L^m) \otimes H^0(X, L^\ell) \longrightarrow H^0(X, \sF \otimes L^{m+\ell})
$$
is surjective, and $\sF$ is $(m+\ell)$-regular with respect to $L$ for every $\ell \geq 0$.

\begin{lemma}[{cf. \cite[Corollary 3.2]{Arapura}}]\label{lem:reg=>res}
If $\sO_X$ is $k$-regular with $k \geq 1$ and $\sF$ is $m$-regular with respect to $L$, then there are finite dimensional vector spaces $W_N, \ldots, W_1, W_0$ over $\mathbf{k}$ and a resolution of $\sF$ of the form
$$
W_N \otimes L^{-m-Nk} \longrightarrow \cdots \longrightarrow W_1 \otimes L^{-m-k} \longrightarrow W_0 \otimes L^{-m} \longrightarrow \sF \longrightarrow 0.
$$
\end{lemma}

\begin{proof}
By  \cite[Theorem 1.8.5]{positivity}, $\sF \otimes L^m$ is globally generated. Letting $W_0:=H^0(X, \sF \otimes L^m)$, we have a short exact sequence
$$
0 \longrightarrow M_0 \longrightarrow W_0 \otimes L^{-m} \longrightarrow \sF \longrightarrow 0.
$$
By  \cite[Theorem 1.8.5]{positivity}, the map
$$
W_0 \otimes H^0(X, L^{(m+k)-m-1}) \longrightarrow H^0(X, \sF \otimes L^{(m+k)-1})
$$
is surjective. Note that
$$
H^{i}(X, L^{(m+k)-m-i}) =0~~\text{ for $i \geq 1$}~~and~~H^{i-1}(X, \sF \otimes L^{(m+k)-i}) = 0~~\text{ for $i \geq 2$}.
$$
Thus $M_0$ is $(m+k)$-regular with respect to $L$. Replacing $\sF$ by $M_0$ and continuing the arguments, we obtain the lemma.
\end{proof}

\section{Lifting Syzygies from Hypersurfaces}\label{sec:lifting}
The aim of this section is showing how to lift syzygies from hypersurfaces (see Theorem \ref{thm:lifting}). This is the main ingredient of the proof of Theorem \ref{thm:main}. We start by setting notations.
Let $X$ be a smooth projective variety, $B$ be a line bundle on $X$, and $L$ be a very ample line bundle on $X$. Assume that $n:=\dim X \geq 2$. Take a very ample line bundle $H$ on $X$, and suppose that
\begin{equation}\label{eq:regBwrtL}
\begin{array}{l}
H^i(X, B \otimes L^m) = 0~~\text{ for $i>0$ and $m>0$ or $i<n$ and $m<0$};\\[3pt]
H^i(X, B \otimes H \otimes L^m) = H^i(X, B \otimes H^{-1} \otimes L^m) = 0~~\text{ for $1 \leq i \leq n-1$ and $m \in \mathbf{Z}$};\\[3pt]
H^0(X, B \otimes H \otimes L^{m}) = H^0(X, B \otimes H^{-1} \otimes L^m)=0~~\text{ for $m<0$};\\[3pt]
H^n(X, B \otimes H \otimes L^m) = H^n(X, B \otimes H^{-1} \otimes L^m)= 0~~\text{ for $m > 0$}.
\end{array}
\end{equation}
In particular, $R(X, B \otimes H; L)$ and $R(X, B \otimes H^{-1}; L)$ are Cohen--Macaulay. Choose a general member $\overline{X} \in |H|$, and put
$$
\begin{array}{l}
\overline{L}:=L|_{\overline{X}},~~\overline{B}:=B|_{\overline{X}},~~\overline{H}:=H|_{\overline{X}};\\[3pt]
V:=H^0(X, L),~~V':=H^0(X, L \otimes H^{-1}),~~\overline{V}:=H^0(X, \overline{L});\\[3pt]
r:=\dim V -1, v' := \dim V' ,~~\overline{r}:=\dim \overline{V} - 1
\end{array}
$$
so that $r=v' + \overline{r}$.
Fix a splitting $V=V' \oplus \overline{V}$. As in \cite[Lemma 3.12]{EL2}, we get
\begin{equation}\label{eq:splitting1}
\displaystyle \wedge^{p+1} M_{L}|_{\overline{X}} = \bigoplus_{j=0}^{p+1} \wedge^j V' \otimes \wedge^{p+1-j} M_{\overline{L}}~~\text{ for $p \geq 0$}.
\end{equation}
By Proposition \ref{prop:koszulcoh}, we obtain
\begin{equation}\label{eq:splitting2}
\begin{array}{l}
\displaystyle K_{p,q}(X, \overline{B}; L) = \bigoplus_{j=0}^p \wedge^j V' \otimes K_{p-j, q}(\overline{X}, \overline{B}; \overline{L})~~\text{ for $p,q \geq 0$};\\[3pt]
\displaystyle K_{p,q}(X, \overline{B} \otimes \overline{H}; L) =   \bigoplus_{j=0}^p \wedge^j V' \otimes K_{p-j, q}(\overline{X}, \overline{B} \otimes \overline{H}; \overline{L})~~\text{ for $p,q \geq 0$}.
\end{array}
\end{equation}

\medskip

Now, put $S:=\bigoplus_{m \geq 0} S^m V$ and $\overline{S}:=\bigoplus_{m \geq 0} S^m \overline{V}$. Consider the following short exact sequence
\begin{equation}\label{eq:ses1}
0 \longrightarrow B  \longrightarrow B \otimes H \longrightarrow \overline{B} \otimes \overline{H} \longrightarrow 0.
\end{equation}
By (\ref{eq:regBwrtL}), we have an exact sequence of finitely generated graded $S$-modules
$$
0 \longrightarrow R(X, B; L) \longrightarrow R(X, B \otimes H; L) \longrightarrow R(X, \overline{B} \otimes \overline{H}; L) \longrightarrow H^1(X, B) \longrightarrow 0.
$$
Let $\overline{R}(X, \overline{B} \otimes \overline{H}; L)$ be the kernel of the map $R(X, \overline{B} \otimes \overline{H}; L) \longrightarrow H^1(X, B)$, which is a finitely generated graded $\overline{S}$-module. Then we get a short exact sequence of finitely generated graded $S$-modules
$$
0 \longrightarrow R(X, B; L) \longrightarrow R(X, B \otimes H; L) \longrightarrow \overline{R}(X, \overline{B} \otimes \overline{H}; L) \longrightarrow 0.
$$
By \cite[Corollary (1.d.4)]{Green}, this induces a connecting map
\begin{equation}\label{eq:theta}
\theta_{p, q} \colon \underbrace{\overline{K}_{p+1, q-1}(X, \overline{B} \otimes \overline{H}; L)}_{:=K_{p+1, q-1}( \overline{R}(X, \overline{B} \otimes \overline{H}; L) , V)} \longrightarrow K_{p,q}(X, B; L).
\end{equation}
Notice that
\begin{equation}\label{eq:decombarK}
\overline{K}_{p+1, q-1}(X, \overline{B} \otimes \overline{H}; L) = \bigoplus_{j=0}^{p+1} \wedge^j V' \otimes \underbrace{\overline{K}_{p+1-j, q-1}(\overline{X}, \overline{B}; \overline{L})}_{\mathclap{:=K_{p+1-j, q-1}(\overline{R}(X, \overline{B} \otimes \overline{H}; L), \overline{V})}}.
\end{equation}
On the other hand, we also have a short exact sequence of finitely generated graded $S$-modules
$$
0 \longrightarrow \overline{R}(X, \overline{B} \otimes \overline{H}; L) \longrightarrow R(X, \overline{B} \otimes \overline{H}; L) \longrightarrow H^1(X, B) \longrightarrow 0.
$$
Since 
$$
K_{p,q}(H^1(X, B), V) = \begin{cases} \wedge^p V \otimes H^1(X, B) & \text{if $q = 0$}\\ 0 & \text{if $q \geq 1$},\end{cases}
$$
we get an exact sequence\\[-20pt]

\begin{small}
\begin{equation}\label{eq:psi1}
\begin{array}{l}
0 \longrightarrow \overline{K}_{p+1, 0}(X, \overline{B} \otimes \overline{H}; L) \xrightarrow{~\psi_0~} K_{p+1, 0}(X, \overline{B} \otimes \overline{H}; L)\\[3pt]
~\text{ }~\text{ }~\text{ }~\text{ }~\text{ }~\longrightarrow \wedge^{p+1}V \otimes H^1(X, B) \xrightarrow{~\varphi~} \overline{K}_{p, 1}(X, \overline{B} \otimes \overline{H}; L) \xrightarrow{~\psi_1~} K_{p, 1}(X, \overline{B} \otimes \overline{H}; L) \longrightarrow 0
\end{array}
\end{equation}
\end{small}

\noindent and an isomorphism
\begin{equation}\label{eq:psi2}
\psi_{q-1} \colon \overline{K}_{p+1, q-1}(X, \overline{B} \otimes \overline{H}; L) \longrightarrow K_{p+1, q-1}(X, \overline{B} \otimes \overline{H}; L)~~\text{ for $q \geq 3$}.
\end{equation}
In view of (\ref{eq:decombarK}), there is a map 
$$
\overline{\psi}_{p+1-j, q-1} \colon \overline{K}_{p+1-j, q-1}(\overline{X}, \overline{B} \otimes \overline{H}; \overline{L}) \longrightarrow K_{p+1-j, q-1}(\overline{X}, \overline{B} \otimes \overline{H}; \overline{L})
$$
such that 
$$
\psi_{q-1} = \bigoplus_{j=0}^{p+1} \id_{\wedge^j V'} \otimes \overline{\psi}_{p+1-j, q-1}.
$$

\medskip

Next, consider the following short exact sequence
\begin{equation}\label{eq:ses2}
0 \longrightarrow B \otimes H^{-1} \longrightarrow B \longrightarrow \overline{B} \longrightarrow 0.
\end{equation}
By (\ref{eq:regBwrtL}), we have a short exact sequence of finitely generated graded $S$-modules
$$
0 \longrightarrow R(X, B \otimes H^{-1}; L) \longrightarrow R(X, B; L) \longrightarrow R(X, \overline{B}; L) \longrightarrow 0.
$$
By \cite[Corollary (1.d.4)]{Green}, this induces a restriction map
\begin{equation}\label{eq:theta'}
\theta_{p,q}' \colon K_{p,q}(X, B; L) \longrightarrow K_{p,q}(X, \overline{B}; L).
\end{equation}

\begin{theorem}\label{thm:lifting}
Fix an index $q \geq 1$. Then we have the following:\\[3pt]
$(1)$ Suppose that the map $\theta_{p,q}$ in (\ref{eq:theta}) is a nonzero map for $p=c$ with $0 \leq c \leq r- \overline{r}-1$. Then the map $\theta_{p,q}$ is a nonzero map for $c \leq p \leq r - \overline{r}-1$, and consequently, 
$$
K_{p,q}(X, B; L) \neq 0~~\text{ for $c \leq p \leq r - \overline{r}-1$}.
$$
$(2)$ Suppose that the map $\theta_{p,q}'$ in (\ref{eq:theta'}) is a nonzero map for $p=r-c'$ with $0 \leq c' \leq \overline{r}$. Then $\theta_{p,q}'$ is a nonzero map for $\overline{r} \leq p \leq r-c'$, and consequently, 
$$
K_{p,q}(X, B; L) \neq 0~~\text{ for $\overline{r} \leq p \leq r - c'$}.
$$
\end{theorem}

\begin{proof}
$(1)$ There is $\alpha_{c+1} \in \overline{K}_{c+1, q-1}(X, \overline{B} \otimes \overline{H}; L)$ such that $\theta_{c,q}(\alpha_{c+1}) \neq 0$. We may assume that 
$$
\alpha_{c+1} = s_1' \wedge \cdots \wedge s_{j_0}' \otimes \alpha' \in \wedge^{j_0} V' \otimes \overline{K}_{c+1-j_0, q-1}(\overline{X}, \overline{B}; \overline{L}) \subseteq \overline{K}_{c+1, q-1}(X, \overline{B} \otimes \overline{H}; L)
$$
for some $0 \leq j_0 \leq c+1$. We proceed by induction on $p$. For $c \leq p-1 \leq r - \overline{r}-2$, we may assume that there is
$$
\alpha_p=s_1' \wedge \cdots \wedge s_j' \otimes \alpha' \in \wedge^j V' \otimes \overline{K}_{c+1-j_0, q-1}(\overline{X}, \overline{B}; \overline{L}) \subseteq \overline{K}_{p, q-1}(X, \overline{B} \otimes \overline{H}; L),
$$
where $j=p-(c+1-j_0)$, such that $\theta_{p-1, q}(\alpha_p) \neq 0$.
Consider the commutative diagram
$$
\xymatrixcolsep{0.69in}
\xymatrix{
\overline{K}_{p+1, q-1}(X, \overline{B} \otimes \overline{H}; L) \ar[r]^-{\theta_{p,q}} \ar[d]_-{\iota} & K_{p,q}(X, B; L) \ar[d]^-{\iota} \\
V \otimes \overline{K}_{p, q-1}(X, \overline{B} \otimes \overline{H}; L) \ar[r]_-{\id_V \otimes \theta_{p-1, q}} & V \otimes K_{p-1, q}(X, B; L).
}
$$
Take any $s_{j+1}' \in V'$ with $s_1' \wedge \cdots \wedge s_j' \wedge s_{j+1}' \neq 0$, and let
$$
\alpha_{p+1}:=s_1' \wedge \cdots \wedge s_{j+1}' \otimes \alpha' \in \wedge^{j+1} V' \otimes \overline{K}_{c+1-j_0, q-1}(\overline{X}, \overline{B}; \overline{L}) \subseteq \overline{K}_{p+1, q-1}(X, \overline{B} \otimes \overline{H}; L).
$$
Then
$$
\iota(\alpha_{p+1}) =\underbrace{\textstyle \sum_{i=1}^{j+1} (-1)^i s_i' \otimes s_1' \wedge \cdots \wedge \widehat{s_i'} \wedge \cdots \wedge s_{j+1}' \otimes \alpha'}_{\text{\tiny $\in V'  \otimes \overline{K}_{p,q-1}(X, \overline{B} \otimes \overline{H}; L)$}} + \underbrace{s_1' \wedge \cdots \wedge s_{j+1}' \otimes \iota(\alpha')}_{\text{\tiny $\in \overline{V} \otimes \overline{K}_{p,q-1}(X, \overline{B} \otimes \overline{H}; L)$}}.
$$
Notice that
$$
\id_V \otimes \theta_{p-1, q}\big( (-1)^{j+1} s_{j+1}' \otimes \alpha_p \big) = (-1)^{j+1} s_{j+1}' \otimes \theta_{p-1, q}(\alpha_p) \neq 0
$$
in $\langle s_{j+1}' \rangle \otimes K_{p-1, q}(X, B; L)$. Observe then that all other terms in $\iota(\alpha_{p+1})$ go into complements of  $\langle s_{j+1}' \rangle \otimes K_{p-1, q}(X, B; L)$ via the map $\id_V \otimes \theta_{p-1, q}$. Thus
$$
(\iota \circ \theta_{p,q}) (\alpha_{p+1})= ((\id_V \otimes \theta_{p-1, q}) \circ \iota)(\alpha_{p+1}) \neq 0,
$$
and hence, $\theta_{p,q}(\alpha_{p+1}) \neq 0$.

\medskip

\noindent $(2)$ There is $\beta_{r-c'} \in K_{r-c', q}(X, B; L)$ such that $\theta_{r-c',q}'(\beta_{r-c'})$ has a nonzero term
$$
s_1' \wedge \cdots \wedge s_{j_0}' \otimes \beta' \in \wedge^{j_0} V' \otimes K_{r-c'-j_0, q}(\overline{X}, \overline{B}; \overline{L}) \subseteq K_{r-c', q}(X, \overline{B}; L)
$$
for some $r-\overline{r}-c' \leq j_0 \leq r-\overline{r}$. We proceed by reverse induction on $p$. For $\overline{r}+1 \leq p+1 \leq r-c'$, we may assume that there is $\beta_{p+1} \in K_{p+1, q}(X, B; L)$ such that $\theta_{p+1, q}'(\beta_{p+1})$ has a nonzero term
$$
s_1' \wedge \cdots \wedge s_{j+1}' \otimes \beta' \in \wedge^{j+1} V' \otimes K_{r-c'-j_0, q}(\overline{X}, \overline{B}; \overline{L}) \subseteq K_{p+1, q}(X, \overline{B}; L),
$$
where $j+1 = p+1-(r-c'-j_0)$. Consider the commutative diagram
$$
\xymatrixcolsep{0.69in}
\xymatrix{
K_{p+1, q}(X, B; L) \ar[r]^-{\theta_{p+1, q}'} \ar[d]_-{\iota} & K_{p+1, q}(X, \overline{B}; L) \ar[d]^-{\iota} \\
V \otimes K_{p,q}(X, B; L) \ar[r]_-{\id_V \otimes \theta_{p,q}'} & V \otimes K_{p,q}(X, \overline{B}; L).
}
$$
We have
$$
\iota(s_1' \wedge \cdots \wedge s_{j+1}' \otimes \beta') =\underbrace{\textstyle \sum_{i=1}^{j+1} (-1)^i s_i' \otimes s_1' \wedge \cdots \wedge \widehat{s_i'} \wedge \cdots \wedge s_{j+1}' \otimes \beta'}_{\text{\tiny $\in V'  \otimes \wedge^j V' \otimes K_{r-c'-j_0,q}(\overline{X}, \overline{B}; \overline{L}) \subseteq V' \otimes K_{p,q}(X, \overline{B}; L)$}} + \underbrace{s_1' \wedge \cdots \wedge s_{j+1}' \otimes \iota(\beta')}_{\text{\tiny $\in \overline{V} \otimes K_{p,q}(X, \overline{B}; L)$}}.
$$
Note that all terms of $\theta_{p+1, q}'(\beta_{p+1})$ not in $\wedge^{j+1} V' \otimes \langle \beta' \rangle$ go into complements of $V' \otimes \wedge^j V' \otimes \langle \beta' \rangle$ via the map $\iota$. Observe then that the term
$$
(-1)^{j+1} s_{j+1}' \otimes s_1' \wedge \cdots \wedge s_j' \otimes \beta' \in V' \otimes \wedge^j V' \otimes K_{r-c'-j_0, q}(\overline{X}, \overline{B}; \overline{L}) \subseteq V \otimes K_{p,q}(X, \overline{B}; L)
$$
of $(\iota \circ \theta_{p+1, q}')(\beta_{p+1})$ cannot be cancelled in $V' \otimes \wedge^j V' \otimes \langle \beta' \rangle$. Thus there is $\beta_p \in K_{p,q}(X, B; L)$ such that $(\id_V \otimes \theta_{p,q}')((-1)^{j+1} s_{j+1}' \otimes \beta_p)$ has the above term, so $\theta_{p,q}'(\beta_p)$ has a nonzero term
$$
s_1' \wedge \cdots \wedge s_j' \otimes \beta' \in \wedge^j V' \otimes K_{r-c'-j_0, q}(\overline{X}, \overline{B}; \overline{L}) \subseteq K_{p, q}(X, \overline{B}; L).
$$
We complete the proof.
\end{proof}

\section{Precise Asymptotic Nonvanishing Theorem}\label{sec:proof}
After establishing key steps as propositions, we finish the proof of Theorem \ref{thm:main} at the end of this section.
We start by setting notations. Let $X$ be a smooth projective variety of dimension $n \geq 1$, and $B$ be a line bundle on $X$. For an integer $d \geq 1$, set
$$
L_d:=\sO_X(dA+P)~~\text {and }~~r_d:=h^0(X, L_d)-1,
$$
where $A$ is an ample divisor and $P$ is an arbitrary divisor on $X$. We assume throughout that $d$ is sufficiently large so that $L_d$ is a sufficiently positive very ample line bundle and $r_d = \Theta(d^n)$. Furthermore, we have
$$
H^i(X, B \otimes L_d^m) = 0~~\text{ for $i>0$ and $m>0$ or $i<n$ and $m<0$}.
$$
For $1 \leq q \leq n$, let $c_q(d)$ be the number such that
$$
K_{c_q(d), q}(X, B; L_d) \neq 0~~\text{ and }~~K_{p, q}(X, B; L_d) = 0~~\text{ for $0 \leq p \leq c_q(d)-1$},
$$
and $c_q'(d)$ be the number such that
$$
K_{r_d-c_q'(d), q}(X, B; L_d) \neq 0~~\text{ and }~~K_{p,q}(X, B; L_d) = 0~~\text{ for $r_d-c_q'(d)+1 \leq p \leq r_d$}.
$$
If $K_{p,q}(X, B; L_d)=0$ for all $p$, then we set $c_q(d) := r_d+1$ and $c_q'(d) := r_d+1$. We will see in Proposition \ref{prop:Theta} that this cannot happen. Recall from \cite[Theorem 1.1]{Park} that $c_q(x) \geq \Theta(d^{q-1})$.

\begin{lemma}\label{lem:CM}
Assume that $H^i(X, B)=0$ for $1 \leq i \leq n-1$, i.e., $R(X, B; L_d)$ is Cohen--Macaulay.
Fix an index $1 \leq q \leq n$. Then we have the following:\\[3pt]
$(1)$ If $K_{p_0, q}(X, B; L_d) = 0$ for some $p_0 \leq \Theta(d^{q-1})$, then $K_{p,q}(X, B; L_d) = 0$ for $p \leq p_0$.\\[3pt]
$(2)$ If $K_{p_0,q}(X, B; L_d) = 0$ for some $p_0 \geq r_d - \Theta(d^{n-q})$, then $K_{p,q}(X, B; L_d) = 0$ for $p \geq p_0$.
\end{lemma}

\begin{proof}
By Proposition \ref{prop:duality}, \cite[Proposition 5.1]{EL2}, and \cite[Theorem 1.1]{Park}, if $0 \leq q' \leq q-1$, then
$$
K_{p,q'}(X, B; L_d) = K_{r_d-p-n, n+1-q'}(X, \omega_X \otimes B^{-1}; L_d)^{\vee}= 0~~\text{ for $p \geq r_d - \Theta(d^{n-q'})$}.
$$
Thus  $K_{p_0, q'}(X, B; L_d)=0$ for $0 \leq q' \leq q$, so the assertion $(2)$ follows. Now, by Proposition \ref{prop:duality}, the assertion $(2)$ implies the assertion $(1)$.
\end{proof}

\begin{remark}
If $R(X, B; L_d)$ is Cohen--Macaulay, then \cite[Theorem 4.1]{EL2}, \cite[Theorem 1.1]{Park}, and Lemma \ref{lem:CM} imply Theorem \ref{thm:main}. In general, we know from \cite[Theorem 4.1]{EL2} and \cite[Theorem 1.1]{Park} that $c_q(d) = \Theta(d^{q-1})$. Using Boij--S\"{o}derberg theory \cite{BS}, \cite{ES1}, one can show that vanishing and nonvanishing of $K_{p,q}(X, B; L_d)$ do not alternate for a while after $c_q(d)$. This means that
$$
K_{p,q}(X, B; L_d) \neq 0~~\text{ for $c_q(d) \leq p \leq r_d - \Theta(d^{n-1})$}.
$$
However, we will not use this remark in our proof of Theorem \ref{thm:main}.
\end{remark}

\begin{lemma}\label{lem:dimX=1}
Theorem \ref{thm:main} holds when $n=1$.
\end{lemma}

\begin{proof}
When $n=1$, \cite[Proposition 5.1 and Corollary 5.2]{EL2} imply
$$
K_{p,1}(X, B; L_d) \neq 0~~\text{ for $h^0(B) \leq p \leq r_d - h^0(X, \omega_X \otimes B^{-1})-1$}.
$$
This shows that $c_1(d) = \Theta(1)$ and $c_1'(d)= \Theta(1)$. Then Lemma \ref{lem:CM} implies the lemma.
\end{proof}

\medskip

Assume henceforth that $n=\dim X \geq 2$. Let $H$ be a very ample line bundle (independent of $d$) on $X$ such that
$$
H^i(X, B \otimes H) = H^i(X, B \otimes H^{-1}) = 0~~\text{ for $1 \leq i \leq n-1$}.
$$
As $d$ is sufficiently large, we have
\begin{equation}\label{eq:regBotimesHwrtL_d}
H^i(X, B \otimes H \otimes L_d^m) = H^i(X, B \otimes H^{-1} \otimes L_d^m) = 0~~\text{ for $1 \leq i \leq n-1$ and $m \in \mathbf{Z}$}.
\end{equation}
This means that $R(X, B \otimes H; L_d)$ and $R(X, B \otimes H^{-1}; L_d)$ are Cohen--Macaulay. Clearly,
\begin{equation}\label{eq:regBotimesHwrtL_d'}
\begin{array}{l}
H^0(X, B \otimes H \otimes L_d^m) = H^0(X, B \otimes H^{-1} \otimes L_d^m)=0~~\text{ for $m<0$};\\[3pt]
H^n(X, B \otimes H \otimes L_d^m) = H^n(X, B \otimes H^{-1} \otimes L_d^m)=0~~\text{ for $m>0$};
\end{array}
\end{equation}
Choose a general member $\overline{X} \in |H|$, and put
$$
\begin{array}{l}
\overline{L}_d:=L_d|_{\overline{X}},~~\overline{B}:=B|_{\overline{X}},~~\overline{H}:=H|_{\overline{X}};\\[3pt]
V_d:=H^0(X, L_d), ~~V_d':=H^0(X, L_d \otimes H^{-1}),~~\overline{V}_d:=H^0(X, \overline{L}_d);\\[3pt]
r_d := \dim V_d - 1,~~v_d' := \dim V_d' ,~~\overline{r}_d:=\dim \overline{V}_d - 1
\end{array}
$$
so that $r_d = v_d' + \overline{r}_d$ and $\overline{r}_d = \Theta(d^{n-1})$.
As in the previous section, fix a splitting $V_d=V_d' \oplus \overline{V}_d$. Then (\ref{eq:splitting1}) and (\ref{eq:splitting2}) hold. Furthermore, the short exact sequences (\ref{eq:ses1}) and (\ref{eq:ses2}) induce the map $\theta_{p,q}$ in (\ref{eq:theta}) and $\theta_{p,q}'$ in (\ref{eq:theta'}), respectively. In view of \cite[Corollary (1.d.4)]{Green}, they fit into the following exact sequences\\[-20pt]

\begin{footnotesize}
\begin{subequations}
\begin{align}
&K_{p+1, q-1}(X, B \otimes H; L_d) \longrightarrow \overline{K}_{p+1, q-1}(X, \overline{B} \otimes \overline{H}; L_d) \xrightarrow{\theta_{p,q}} K_{p,q}(X, B; L_d) \longrightarrow K_{p,q}(X, B \otimes H; L_d);\label{eq:les1}\\
& K_{p,q}(X, B \otimes H^{-1}; L_d) \longrightarrow K_{p,q}(X, B; L_d) \xrightarrow{\theta_{p,q}'} K_{p,q}(X, \overline{B}; L_d) \longrightarrow K_{p-1, q+1}(X, B \otimes H^{-1}; L_d).\label{eq:les2}
\end{align}
\end{subequations}
\end{footnotesize}

\begin{proposition}\label{prop:Theta}
For each $1 \leq q \leq n$, we have
$$
c_q(d) = \Theta(d^{q-1})~~\text{ and }~~c_q'(d) = \begin{cases} \Theta(d^{n-q}) & \text{if $H^{q-1}(X, B) =0$ or $q=1$} \\  q-1 & \text{if $H^{q-1}(X, B) \neq 0$ and $q \geq 2$}. \end{cases}
$$
\end{proposition}

\begin{proof}
We proceed by induction on $n$. As the assertion holds for $n=1$ by Lemma \ref{lem:dimX=1}, we assume that $n \geq 2$ and the assertions of the lemma hold for $\overline{X}$.

\medskip

First, we consider $c_q'(d)$. Suppose that $H^{q-1}(X, B) \neq 0$ and $q \geq 2$. Then 
$$
K_{r_d-q+1, q}(X, B; L_d) = H^{q-1}(X, \wedge^{r_d}M_{L_d} \otimes B \otimes L_d) = H^{q-1}(X, B) \neq 0.
$$
Since $K_{p, q}(X, B; L_d) = 0$ for $p \geq r_d-q$, it follows that $c_q'(d) = q-1$. 
Suppose that $H^{q-1}(X, B) = 0$ when $q \geq 2$ or $q=1$. By Proposition \ref{prop:duality} and \cite[Theorem 1.1]{Park}, 
$$
K_{p,q}(X, B; L) \subseteq K_{r_d-p-n, n+1-q}(X, \omega_X \otimes B^{-1}; L_d)^{\vee} = 0~~\text{ for $0 \leq r_d-p-n \leq \Theta(d^{n-q})$},
$$
so $c_q'(d) \geq \Theta(d^{n-q})$. For $2 \leq q \leq n$, let $\overline{c}_{q-1}'(d)$ be the number such that
$$
K_{\overline{r}_d-\overline{c}_{q-1}'(d), q-1}(\overline{X}, \overline{B} \otimes \overline{H}; \overline{L}_d) \neq 0~~\text{ and }~~K_{p,q-1}(\overline{X}, \overline{B} \otimes \overline{H}; \overline{L}_d) = 0~~\text{ for $p \geq \overline{r}_d-\overline{c}_{q-1}'(d)+1$}.
$$
By induction, $\overline{c}_{q-1}'(d) = \Theta(d^{n-q})$. Recall that $H^1(X, B) =0$ when $q=2$. By considering (\ref{eq:splitting2}), (\ref{eq:psi1}), (\ref{eq:psi2}), we get
$$
\overline{K}_{r_d - \overline{c}_{q-1}'(d), q-1}(X, \overline{B} \otimes \overline{H}; L_d) = K_{r_d - \overline{c}_{q-1}'(d), q-1}(X, \overline{B} \otimes \overline{H}; L_d)  \neq 0~~\text{ for $2 \leq q \leq n$}.
$$
Possibly replacing $H$ by more positive $H$ (still independent of $d$), we may assume that
\begin{equation}\label{eq:h^0(B+H)>h^0(B)}
h^0(X, B \otimes H) > h^0(X, B).
\end{equation}
Then $\overline{K}_{0,0}(\overline{X}, \overline{B} \otimes \overline{H}; \overline{L}_d) \neq 0$. Putting $\overline{c}'_0(d):=\overline{r}_d = \Theta(d^{n-1})$, we get from (\ref{eq:decombarK}) that
$$
\overline{K}_{r_d - \overline{c}_{0}'(d), 0}(X, \overline{B} \otimes \overline{H}; L_d)  \neq 0.
$$
For $1 \leq q \leq n$, thanks to (\ref{eq:regBotimesHwrtL_d}), Proposition \ref{prop:duality} and \cite[Theorem 1.1]{Park} yield that
$$
K_{r_d - \overline{c}_{q-1}'(d), q-1}(X, B \otimes H; L_d) = K_{ \overline{c}_{q-1}'(d) -n, n+2-q}(X, \omega_X \otimes B^{-1} \otimes H^{-1}; L_d)^{\vee} = 0, 
$$
since $\overline{c}_{q-1}'(d)-n = \Theta(d^{n-q}) < \Theta(d^{n+1-q})$. Then the map
$$
\theta_{r_d - \overline{c}_{q-1}'(d)-1,q} \colon \overline{K}_{r_d - \overline{c}_{q-1}'(d), q-1}(X, \overline{B} \otimes \overline{H}; L_d) \longrightarrow K_{r_d - \overline{c}_{q-1}'(d)-1,q}(X, B; L_d)
$$
in (\ref{eq:les1}) is a nonzero injective map. Thus $c_q'(d) \leq \overline{c}_{q-1}'(d)+1 = \Theta(d^{n-q})$, so $c_q'(d) = \Theta(d^{n-q})$.

\medskip

Next, we consider $c_q(d)$. We know from \cite[Theorem 1.1]{Park} that $c_q(d) \geq \Theta(d^{q-1})$. When $q=n$, Proposition \ref{prop:duality} says that there is a surjective map
$$
K_{p, n}(X, B; L_d) \longrightarrow K_{r_d-p-n, 1}(X, \omega_X \otimes B^{-1}; L_d)^{\vee}.
$$
As we have seen in the previous paragraph that $K_{r_d-p-n, 1}(X, \omega_X \otimes B^{-1}; L_d)^{\vee} \neq 0$ for some $p=\Theta(d^{n-1})$, we have  $c_n(d) \leq \Theta(d^{n-1})$. Hence $c_n(d) = \Theta(d^{n-1})$. Assume that $1 \leq q \leq n-1$. Let $\overline{c}_q(d)$ be the number such that
$$
K_{\overline{c}_q(d), q}(\overline{X}, \overline{B}; \overline{L}_d) \neq 0~~\text{ and }~~K_{p, q}(\overline{X}, \overline{B}; \overline{L}_d) = 0~~\text{ for $p \leq \overline{c}_q(d)-1$}.
$$
By induction, $\overline{c}_q(d) = \Theta(d^{q-1})$. Note that
$$
K_{\overline{c}_q(d),q}(X, \overline{B}; L_d) \neq 0
$$
thanks to (\ref{eq:splitting2}). By \cite[Theorem 1.1]{Park}, 
$$
K_{\overline{c}_q(d)-1, q+1}(X, B \otimes H^{-1}; L_d) = 0
$$
since $\overline{c}_q(d)-1 = \Theta(d^{q-1}) < \Theta(d^q)$. Then the map
$$
\theta_{\overline{c}_q(d), q}' \colon K_{\overline{c}_q(d),q}(X, B; L_d) \longrightarrow K_{\overline{c}_q(d),q}(X, \overline{B}; L_d)
$$
in (\ref{eq:les2}) is a nonzero surjective map. Thus $c_q(d) \leq \overline{c}_q(d) = \Theta(d^{q-1})$, so $c_q(d) = \Theta(d^{q-1})$. 
\end{proof}

Next, we prove the following technical lemma.

\begin{lemma}\label{lem:technical2}
Let $B'$ be a line bundle on $X$ (independent of $d$), and $L$ be a very ample line bundle on $X$ (independent of $d$) such that
\begin{align*}
&H^i(X, L^m) = 0~~\text{ for $i>0$ and $m>0$ or $i<n$ and $m<0$};\\
&H^i(X, B' \otimes L^m)=0~~\text{ for $i>0$ and $m>0$ or $i<n$ and $m<0$};\\
&H^i(X, M_{L_d} \otimes L^m) = 0~~\text{ for $i>0$ and $m>0$}.
\end{align*}
Put  $H:=L^{n+1}$. For $1 \leq p \leq \Theta(d^{q-1})$ and $1 \leq q \leq n+1$, we have the following:\\[3pt]
$(1)$ If $q \geq 2$ and $H^{q-1}(X, \wedge^p M_{L_d} \otimes B' \otimes L_d)=0$, then $H^{q-1}(X, \wedge^{p+1} M_{L_d} \otimes B' \otimes H \otimes L_d)=0$.\\[3pt]
$(2)$ If $q \leq n$ and $H^q(X, \wedge^p M_{L_d} \otimes B')=0$, then $H^q(X, \wedge^{p+1} M_{L_d} \otimes B' \otimes H)=0$.
\end{lemma}

\begin{proof}
Notice that $\sO_X$ and $M_{L_d}$ are $(n+1)$-regular with respect to $L$. By Lemma \ref{lem:reg=>res}, there are finitely dimensional vector spaces $\ldots, W_1, W_0$ over $\mathbf{k}$ and an exact sequence
\begin{equation}\label{eq:resM_{L_d}}
\cdots \longrightarrow  W_1 \otimes H^{-2} \longrightarrow  W_0 \otimes H^{-1} \longrightarrow M_{L_d} \longrightarrow 0.
\end{equation}

\noindent $(1)$ By Lemma \ref{lem:technical1}, it is sufficient to prove that
$$
H^{q-1}(X, M_{L_d} \otimes \wedge^p M_{L_d} \otimes B' \otimes H \otimes L_d) = 0.
$$
For this purpose, consider the exact sequence from (\ref{eq:resM_{L_d}}):
$$
\begin{array}{l}
\cdots \longrightarrow  W_1 \otimes \wedge^{p}M_{L_d} \otimes B' \otimes H^{-1} \otimes L_d \longrightarrow  W_0 \otimes \wedge^{p}M_{L_d} \otimes B' \otimes L_d\\[3pt]
~\text{ }~\text{ }~\text{ }~\text{ }~\text{ }~\text{ }~\text{ }~\text{ }~\text{ }~\text{ }~\text{ }~\text{ }~\text{ }~\text{ }~\text{ }~\text{ }~\text{ }~\text{ }~\text{ }~\text{ }~\text{ }~\text{ }~\text{ }~\text{ }~\text{ }~\text{ }~\text{ } \longrightarrow M_{L_d} \otimes \wedge^{p}M_{L_d} \otimes B' \otimes H \otimes L_d \longrightarrow 0.
\end{array}
$$
In view of \cite[Proposition B.1.2]{positivity}, it suffices to show that
$$
H^{q-1+i}(X, \wedge^{p}M_{L_d} \otimes B' \otimes H^{-i} \otimes L_d) =  K_{p-q-i+1, q+i}(X, B' \otimes H^{-i}; L_d).
=0~~\text{ for $i \geq 0$}.
$$
When $i=0$, this is the given condition. When $i \geq 1$, this follows from \cite[Theorem 1.1]{Park} since $p-q-i+1 \leq \Theta(d^{q-1}) < \Theta(d^{q+i-1})$.

\medskip

\noindent $(2)$ By Lemma \ref{lem:technical1}, it is sufficient to prove that
\begin{equation}\label{eq:claimtechnical}
H^q(X, M_{L_d} \otimes \wedge^{p} M_{L_d} \otimes B' \otimes H)=0.
\end{equation}
Let $M_0$ be the kernel of the map $W_0 \otimes H^{-1} \to M_{L_d}$ in (\ref{eq:resM_{L_d}}). We have a short exact sequence
\begin{equation}\label{eq:sesM_0}
0 \longrightarrow M_0 \otimes B' \otimes H \longrightarrow W_0 \otimes B' \longrightarrow M_{L_d} \otimes B' \otimes H \longrightarrow 0.
\end{equation}
As $H^q(X, \wedge^p M_{L_d} \otimes B')=0$, the claim (\ref{eq:claimtechnical}) is implied by the injectivity of the map
$$
\rho \colon H^{q+1}(X, \wedge^p M_{L_d} \otimes B' \otimes H)  \longrightarrow W_0 \otimes  H^{q+1}(X, \wedge^p M_{L_d} \otimes B').
$$
This map fits into the following commutative diagram
$$
\xymatrixcolsep{0.3in}
\xymatrix{
H^q (X, \wedge^{p} M_{L_d} \otimes M_0 \otimes B' \otimes H \otimes L_d) \ar[r] \ar[d] & W_0 \otimes H^{q}(X, \wedge^{p} M_{L_d} \otimes B' \otimes L_d) \ar[d] \\
H^{q+1}(X, \wedge^p M_{L_d} \otimes M_0 \otimes B' \otimes H) \ar[r]^-{\rho} \ar[d]_-{\psi} & W_0 \otimes  H^{q+1}(X,  \wedge^p M_{L_d} \otimes B') \ar[d] \\
\wedge^p V_d \otimes H^{q+1}(X, M_0 \otimes B' \otimes H) \ar[r]^-{\varphi} & W_0 \otimes \wedge^p V_d \otimes H^{q+1}(X, B').
}
$$
For the injectivity of the map $\rho$, it is enough to check that $\psi$ and $\varphi$ are injective. From (\ref{eq:resM_{L_d}}), we have an exact sequence
$$
\cdots \longrightarrow W_2 \otimes H^{-3} \longrightarrow W_1 \otimes H^{-2} \longrightarrow M_0 \longrightarrow 0.
$$
By \cite[Theorem 1.1]{Park},
$$
H^{q+i}(X,  \wedge^p M_{L_d} \otimes B' \otimes H^{-i-1} \otimes L_d) = K_{p-q-i, q+i+1}(X, B' \otimes H^{-i-1}; L_d) = 0~~\text{ for $i \geq 0$}
$$
 since $p-q-i \leq \Theta(d^{q-1}) < \Theta(d^{q+i})$. By \cite[Proposition B.1.2]{positivity}, 
$$
H^{q}(X, \wedge^{p} M_{L_d}  \otimes M_0 \otimes B' \otimes H \otimes L_d) = 0,
$$
so $\psi$ is injective. On the other hand, we get from (\ref{eq:sesM_0}) that
$$
H^{n+2-q}(X, M_0 \otimes B' \otimes H) = W_0 \otimes H^{n+2-q}(X, B'),
$$
so $\varphi$ is an isomorphism. 
\end{proof}

Now, take a very ample line bundle $L$ on $X$ (independent of $d$) such that
\begin{subequations}
\begin{align}
&H^i(X, L^m) = 0~~\text{ for $i>0$ and $m>0$ or $i<n$ and $m<0$};\label{eq:regwrtL1}\\
&H^i(X, B \otimes L^m)=0~~\text{ for $i>0$ and $m>0$ or $i<n$ and $m<0$};\label{eq:regwrtL2}\\
&H^i(X, M_{L_d} \otimes L^m) = 0~~\text{ for $i>0$ and $m>0$};\label{eq:regwrtL3}\\
&H^i(X, M_{L_d} \otimes \omega_X \otimes B^{-1} \otimes L^m) = 0~~\text{ for $i>0$ and $m>0$};\label{eq:regwrtL3.5}
\end{align}
\end{subequations}
By Proposition \ref{prop:Theta}, we can take an integer $c \geq c_1(d)+1, c_n'(d)-n+1$ independent of $d$. Successively applying Lemma \ref{lem:technical2} and possibly replacing $L$ by a higher power of $L$ (still independent of $d$), we may assume that
$$
\begin{array}{l}
K_{c-1, 1}(X, B \otimes L^{n+1}; L_d) = H^1(X, \wedge^c M_{L_d} \otimes B \otimes L^{n+1}) = 0;\\[3pt]
K_{c-1, 1}(X, \omega_X \otimes B^{-1} \otimes L^{n+1}; L_d) = H^1(X, \wedge^c M_{L_d} \otimes \omega_X \otimes B^{-1} \otimes L^{n+1}) = 0.
\end{array}
$$
By Lemma \ref{lem:CM}, we have\\[-23pt]

\begin{small}
\begin{subequations}
\begin{align}
&K_{c_1(d), 1}(X, B \otimes L^{n+1}; L_d) = H^1(X, \wedge^{c_1(d)+1} M_{L_d} \otimes B \otimes L^{n+1}) = 0;\label{eq:regwrtL4}\\
&K_{c_n'(d)-n, 1}(X, \omega_X \otimes B^{-1} \otimes L^{n+1}; L_d) = H^1(X, \wedge^{c_n'(d)-n+1} M_{L_d} \otimes \omega_X \otimes B^{-1} \otimes L^{n+1}) = 0.\label{eq:regwrtL5}
\end{align}
\end{subequations}
\end{small}

\noindent From now on, replace $H$ by $H:=L^{n+1}$. Then (\ref{eq:regBotimesHwrtL_d}) holds thanks to (\ref{eq:regwrtL2}), so $R(X, B \otimes H; L_d)$ and $R(X, B \otimes H^{-1}; L_d)$ are Cohen--Macaulay. Clearly, (\ref{eq:regBotimesHwrtL_d'}) is satisfied.

\begin{proposition}\label{prop:theta}
Assume that $n \geq 2$. For $1 \leq q \leq n$, we have the following:\\[3pt]
$(1)$ If $K_{p,q}(X, B; L_d) = 0$ for $p < c_q(d)$, then $K_{p+1, q}(X, B \otimes H; L_d) = 0$. Consequently, the map
$$
\theta_{c_q(d),q} \colon \overline{K}_{c_q(d)+1, q-1}(X, \overline{B} \otimes \overline{H}; L_d) \longrightarrow K_{c_q(d),q}(X, B; L_d)
$$
in (\ref{eq:les1}) is a nonzero surjective map.\\[3pt]
$(2)$ If $K_{p,q}(X, B; L_d) = 0$ for $p > r_d - c_q'(d)$, then $K_{p-1, q}(X, B \otimes H^{-1}; L_d) = 0$. Consequently, the map
$$
\theta_{r_d-c_q'(d), q}' \colon K_{r_d-c_q'(d), q}(X, B; L_d) \longrightarrow K_{r_d-c_q'(d), q}(X, \overline{B}; L_d)
$$
in (\ref{eq:les2}) is a nonzero injective map.
\end{proposition}

\begin{proof}
\noindent $(1)$ Recall that $R(X, B \otimes H; L_d)$ is Cohen--Macaulay. Then Lemma \ref{lem:CM} says that
$$
K_{c_q(d), q}(X, B \otimes H; L_d) = 0~~\Longrightarrow~~K_{p+1, q}(X, B \otimes H; L_d) = 0~~\text{ for $p \leq c_q(d)-1$}.
$$ 
Thus it suffices to show that $K_{c_q(d), q}(X, B \otimes H; L_d) = 0$. The case $q=1$ is nothing but (\ref{eq:regwrtL4}). Assume that $2 \leq q \leq n$. Note that the given condition is
$$
H^{q-1}(X, \wedge^{c_q(d)+q-2} M_{L_d} \otimes B \otimes L_d) = K_{c_q(d)-1, q}(X, B; L_d) = 0.
$$
Then Lemma  \ref{lem:technical2} $(1)$ yields
$$
K_{c_q(d), q}(X, B \otimes H; L_d) = H^{q-1}(X,  \wedge^{c_q(d)+q-1}M_{L_d} \otimes B \otimes H \otimes L_d) = 0.
$$

\medskip

\noindent $(2)$ Recall that $R(X, B \otimes H^{-1}; L_d)$ is Cohen--Macaulay. Then Lemma \ref{lem:CM} says that
$$
K_{r_d-c_q'(d), q}(X, B \otimes H^{-1}; L_d) =0~~\Longrightarrow~~K_{p-1, q}(X, B \otimes H^{-1}; L_d) = 0~~\text{ for $p \geq r_d - c_q'(d)+1$}.
$$
Thus it suffices to show that $K_{r_d-c_q'(d), q}(X, B \otimes H^{-1}; L_d) = 0$. 
By Proposition \ref{prop:duality},
$$
K_{r_d-c_q'(d), q}(X, B \otimes H^{-1}; L_d) = K_{c_q'(d)-n, n+1-q}(X, \omega_X \otimes B^{-1} \otimes H; L_d)^{\vee}.
$$
We need to show that
\begin{equation}\label{eq:K_{p+1,q}=0(2)}
H^{n+1-q}(X, \wedge^{c_q'(d)-q+1} M_{L_d} \otimes \omega_X \otimes B^{-1} \otimes H) = 0.
\end{equation}
When $q=n$, (\ref{eq:K_{p+1,q}=0(2)}) is the same to (\ref{eq:regwrtL5}). Assume that $1 \leq q \leq n-1$. 
If $q \geq 2$ and $H^{q-1}(X, B) \neq 0$, then $c_q'(d) = q-1$ so that (\ref{eq:K_{p+1,q}=0(2)}) holds by (\ref{eq:regwrtL2}). Assume $H^{q-1}(X, B) =0$ when $q \geq 2$. Then $c_q'(d) = \Theta(d^{n-q})$. The given condition and Serre duality yield 
$$
\begin{array}{rcl}
H^{n+1-q}(X, \wedge^{c_q'(d)-q} M_{L_d} \otimes \omega_X \otimes B^{-1}) &=& H^{n+1-q}(X, \wedge^{r_d-c_q'(d)+q} M_{L_d}^{\vee} \otimes \omega_X \otimes B^{-1} \otimes L_d^{-1}) \\[3pt]
&=& H^{q-1}(X, \wedge^{r_d-c_q'(d)+q} M_{L_d} \otimes B \otimes L_d)^{\vee} \\[3pt]
&=& K_{r_d-c_q'(d)+1, q}(X, B; L_d)^{\vee} ~=~0.
\end{array}
$$
Then the claim (\ref{eq:K_{p+1,q}=0(2)}) follows from Lemma \ref{lem:technical2} $(2)$.
\end{proof}

Theorem \ref{thm:main} now follows at once from the previous propositions and Theorem \ref{thm:lifting}.

\begin{proof}[Proof of Theorem \ref{thm:main}]
By Lemma \ref{lem:dimX=1}, we may assume that $n \geq 2$. Note that (\ref{eq:main1}) is proved in Proposition \ref{prop:Theta}. For $1 \leq q \leq n$, by Theorem \ref{thm:lifting} $(1)$ and Proposition \ref{prop:theta} $(1)$,
$$
K_{p,q}(X, B; L_d) \neq 0~~\text{ for $c_q(d) \leq p \leq r_d - \overline{r}_d -1$}.
$$
On the other hand, \cite[Proposition 5.1]{EL2} says $K_{p,0}(X, B; L_d) = 0$ for $p > \Theta(1)$. Thus we obtain (\ref{eq:main2}) for $q=1$.
For $2 \leq q \leq n$, by Theorem \ref{thm:lifting} $(2)$ and Proposition \ref{prop:theta} $(2)$,
$$
K_{p,q}(X, B; L_d) \neq 0~~\text{ for $\overline{r}_d \leq p \leq r_d - c_q'(d)$}.
$$
As $r_d - \overline{r}_d - 1 = \Theta(d^n) > \Theta(d^{n-1}) = \overline{r}_d$, we obtain  (\ref{eq:main2}) for $2 \leq q \leq n$.
\end{proof}

\section{Complements and Problems}\label{sec:openprob}
In this section, we show some additional results, and discuss some open problems. Recall that the asymptotic vanishing theorem (\cite[Theorem 1.1]{Park}) holds for singular varieties with coherent sheaves. Precisely, let $X$ be a projective variety of dimension $n$, and $B$ be a coherent sheaf on $X$. For an integer $d \geq 1$, let $L_d:=\sO_X(dA+P)$, where $A$ is an ample divisor and $P$ is an arbitrary divisor on $X$. For each $1 \leq q \leq n+1$,  if $d$ is sufficiently large, then
$$
K_{p,q}(X, B; L_d) =0~~\text{ for $0 \leq p \leq \Theta(d^{q-1})$}.
$$
We expect that Theorem \ref{thm:main} also holds in this setting.

\begin{conjecture}
Theorem \ref{thm:main} still holds when $X$ is a projective variety and $B$ is a coherent sheaf on $X$ with $\Supp B = X$.
\end{conjecture}

Note that the expected nonvanishing of $K_{p,q}(X, B; L_d)$ for $q >\dim \Supp B + 1$ may not hold.

\begin{remark}
In the proof of Theorem \ref{thm:main}, we use the assumption that $X$ is smooth and $B$ is a line bundle only when we apply Serre duality. Thus Theorem \ref{thm:main} holds when $X$ is Cohen--Macaulay and $B$ is a vector bundle.
\end{remark}

From now on, we assume that $X$ is smooth and $B$ is a line bundle as in Theorem \ref{thm:main}. In the remaining, fix an index $1 \leq q \leq n$. It is very natural to study the asymptotic growth of $c_q(d)$ and $c_q'(d)$ as $d \to \infty$. In the spirit of \cite{Zhou}, we give an effective upper bound for each of $c_q(d)$ and $c_q'(d)$. For this purpose, we introduce some notations. Choose suitably positive very ample divisors $H_1, \ldots, H_{n-1}$ on $X$ such that
$$
\overline{X}_i:=H_1 \cap \cdots \cap H_i
$$
is a smooth projective variety for every $0 \leq i \leq n-1$. Note that $\overline{X}_0=X$. For each $0 \leq i \leq n-1$, put 
$$
\begin{array}{l}
\overline{H}_i:=\sO_X(H_{i+1})|_{\overline{X}_i}, ~\overline{B}_i:=B|_{\overline{X}_i},~\overline{B}_i':=B(H_1+\cdots+H_i)|_{\overline{X}_i};\\[3pt]
\overline{L}_d:=L_d|_{\overline{X}_i}, ~\overline{r}_i(d):=h^0(\overline{X}_i, \overline{L}_d)-1 = \Theta(d^{n-i}), ~\overline{r}_n(d):=0,
\end{array}
$$
and assume that (\ref{eq:regBotimesHwrtL_d}), (\ref{eq:regBotimesHwrtL_d'}),  (\ref{eq:h^0(B+H)>h^0(B)}) hold for 
$$
X=\overline{X}_i, ~B=\overline{B}_i, \overline{B}_i', \omega_{X_i} \otimes \overline{B}_i^{-1}, ~H=\overline{H}_i,~L_d=\overline{L}_d.
$$

\begin{proposition}\label{prop:boundforc_q,c_q'}
$c_q(d) \leq \overline{r}_{n+1-q}(d)-q+1$ and $c_q'(d) \leq \overline{r}_q(d)+q$.
\end{proposition}

\begin{proof}
We proceed by induction on $n=\dim X$.
When $n= 1$, the assertion is trivial. Assume that $n \geq 2$. For $c_q'(d)$, we may assume that $H^{q-1}(X, B) =0$ or $q=1$.
In the proof of Proposition \ref{prop:Theta}, we proved that
 $$
\theta_{r_d- \overline{r}_q(d)-q,q} \colon \overline{K}_{r_d - \overline{r}_q(d)-(q-1), q-1}(X, \overline{B} \otimes \overline{H}; L_d) \longrightarrow K_{r_d- \overline{r}_q(d)-q,q}(X, B; L_d)
$$
is nonzero, so $c_q'(d) \leq \overline{r}_q(d)+q$. We also have $K_{r_d-\overline{r}_1(d)-1, 1}(X, \omega_X \otimes B^{-1}; L_d) \neq 0$. 
By Proposition \ref{prop:duality}, $K_{r_1(d)-n+1, 1}(X, B; L_d) \neq 0$. Thus we obtain $c_n(d) \leq \overline{r}_1(d)-n+1$. For $1 \leq q \leq n-1$, in the proof of Proposition \ref{prop:Theta}, we proved that
$$
\theta_{\overline{r}_{n+1-q}(d)-q+1, q}' \colon K_{\overline{r}_{n+1-q}(d)-q+1,q}(X, B; L_d) \longrightarrow K_{\overline{r}_{n+1-q}(d)-q+1,q}(X, \overline{B}; L_d)
$$
is nonzero, so $c_q(d) \leq \overline{r}_{n+1-q}(d)-q+1$.
\end{proof}

\begin{remark}
In Proposition \ref{prop:boundforc_q,c_q'}, we do not assume that $R(X, B; L_d)$ is Cohen--Macaulay. However, when $R(X, B; L_d)$ is Cohen--Macaulay, by a more careful analysis, one can improve bounds for $c_q(d)$ and $c_q'(d)$ as in \cite{Zhou}. In particular, one can recover \cite[Theorem 6.1]{EL2}: If $X=\nP^n, B=\sO_{\nP^n}(b), L_d=\sO_{\nP^n}(d)$ and $b \geq 0, d \gg 0$, then
\begin{equation}\label{eq:veronsese}
c_q(d) \leq {d+q \choose q} - {d - b- 1 \choose q} - q~~\text{ and }~~c_q'(d) \leq {d+n-q \choose n-q}  - {n+b \choose q+b}  + q.
\end{equation}
We leave the details to interested readers.
\end{remark}

In characteristic zero, David Yang \cite[Theorem 1]{Yang} confirmed that $c_1(d)$ is a constant. This gives an answer to \cite[Problem 7.2]{EL2}. On the other hand, for Veronese syzygies, Ein--Lazarsfeld  \cite[Conjecture 2.3]{EL4} conjectured that equalities hold in (\ref{eq:veronsese}) whenever $d \geq b+q+1$. In particular, $c_q(d)$ and $c_q'(d)$ are polynomials. One may hope that the same is true in general.

\begin{question}[{cf. \cite[Remark 3.2]{EEL2}}]
$(1)$ Does the limit
$$
\lim_{d \to \infty} \frac{c_q(d)}{d^{q-1}}
$$ 
exist? If so, is the function $c_q(d)$ a polynomial of degree $q-1$ for sufficiently large $d$? What can one say about the leading coefficient $a_{q-1}:=\lim_{d \to \infty} c_q(d)/d^{q-1}$ of $c_q(d)$?\\[3pt]
$(2)$ Suppose that $H^{q-1}(X, B) = 0$ if $q \geq 2$. Does the limit
$$
\lim_{d \to \infty} \frac{c_q'(d)}{d^{n-q}}
$$
exist? If so, is the function $c_q'(d)$ a polynomial of degree $n-q$ for sufficiently large $d$? What can one say about the leading coefficient $a_{n-q}':=\lim_{d \to \infty} c_q'(d)/d^{n-q}$ of $c_q'(d)$?
\end{question}

When $\Char(\mathbf{k})=0$, a geometric meaning of the constant $c_1(d)$ was explored as follows. Ein--Lazarsfeld--Yang \cite[Theorem A]{ELY} proved that if $B$ is $p$-jet very ample, then $c_1(d) \geq p+1$. Agostini \cite[Theorem A]{Agostini} proved that if $c_1(d) \geq p+1$, then $B$ is $p$-very ample. These results are higher dimensional generalizations of the gonality conjecture on syzygies of algebraic curves, which was established by Ein--Lazarsfeld \cite{EL3} and Rathmann \cite{Rathmann}. 
On the other hand, Eisenbud--Green--Hulek--Popescu \cite{EGHP} related the nonvanishing of $K_{p,q}(X, L_d)$ to the existence of special secant planes. In particular, if the property $N_k$ holds for $L_d$,
then $L_d$ is $(k+1)$-very ample. It would be exceedingly interesting know whether the nonexistence of special secant planes implies the vanishing of certain $K_{p,q}(X, L_d)$. 

\medskip

When $X$ is a smooth projective curve and $\Char(\mathbf{k}) = 0$, Rathmann \cite[Theorem 1.2]{Rathmann} showed that if $H^1(X, L_d) =0$ and $H^1(X, B^{-1} \otimes L_d) = 0$, then $K_{p,1}(X, B; L_d) = 0$. It is natural to extend this effective result to higher dimensions. 
When $B = \sO_X$ and $L_d=\sO_X(K_X + dA)$, the following problem is closely related to Mukai's conjecture \cite[Conjecture 4.2]{EL1}. 

\begin{problem}
$(1)$ Suppose that $c_q(d)$ is a polynomial of degree $q-1$ for sufficiently large $d$. Find an effective bound for $d_0$ such that $c_q(d)$ becomes a polynomial for $d \geq d_0$.\\[3pt]
$(2)$ Suppose that $H^{q-1}(X, B) = 0$ if $q \geq 2$ and $c_q'(d)$ is a polynomial of degree $n-q$ for sufficiently large $d$. Find an effective bound for $d_0'$ such that $c_q'(d)$ becomes a polynomial for $d \geq d_0$.
\end{problem}

Now, we turn to the asymptotic behaviors of the \emph{Betti numbers}
$$
\kappa_{p,q}(X, B; L_d):=\dim K_{p,q}(X, B; L_d).
$$
Ein--Erman--Lazarsfeld conjectured that the Betti numbers $\kappa_{p,q}(X, B; L_d)$ are normally distributed \cite[Conjecture B]{EEL1}, and they verified the conjecture for curves \cite[Proposition A]{EEL1}. The normal distribution conjecture suggests the following unimodality conjecture. 

\begin{conjecture}
The Betti numbers $\kappa_{p,q}(X, B; L_d)$ form a unimodal sequence.
\end{conjecture}

As the cases of very small $p$ and very large $p$ for $\kappa_{p,q}(X, B; L_d)$ are negligible in the normal distribution conjecture, the unimodality conjecture is not a consequence of the normal distribution conjecture. 
Finally, we verify the unimodality conjecture for curves following the strategy of Erman \cite{Erman} based on  Boij--S\"{o}derberg theory. Eisenbud--Schreyer \cite{ES1} and Boij--S\"{o}derberg \cite{BS} showed that the Betti table of a graded module over a polynomial ring is a positive rational sum of pure diagrams (see \cite[Theorem 2.2]{ES2} for the precise statement). Let $C$ be a smooth projective curve, $B$ be a line bundle, and $L$ be a very ample line bundle of sufficiently large degree $d$. Put $c:=h^0(X, B)$ and $c':=h^1(X, B)$. By \cite[Proposition 5.1 and Corollary 5.2]{EL2}, $\kappa_{p,0}(C, B; L) = 0$ for $p \geq c$ and $\kappa_{p,2}(C, B; L)=0$ for $p \leq r-1-c'$, where $r:=h^0(X, L) -1$. By Riemann--Roch theorem, $r=d-g \approx d$ since $d$ is sufficiently large. Let $\pi$ be the Betti table of $R(C, B; L)$. Then \cite[Theorem 2.2]{ES2} says that
\begin{equation}\label{eq:bsdecomp}
\pi = \sum_{i=0}^{c} \sum_{j=0}^{c'} a_{i,j} \pi_{i,j}~~\text{ for some rational numbers $a_{i,j} \geq 0$ with}~~\sum_{i=0}^c \sum_{j=0}^{c'} a_{i,j} = d,
\end{equation}
where each $\pi_{i,j}$ is the pure diagram of the form:

\begin{center}
\texttt{ \begin{tabular}{l|ccccccccc}
         & $0$ & $\cdots$ &  $i-1$  & $i$   & $\cdots$ & $r-j-1$ & $r-j$ & $\cdots$ & $r-1$ \\ \hline
    $0$    & * & $\cdots$ & *   & -  & $\cdots$  & -  & -  & $\cdots$ & - \\
    $1$   & -  & $\cdots$ & -    & *  & $\cdots$  & *  & -  & $\cdots$  & -\\
    $2$   & -  &  $\cdots$ & - & - & $\cdots$ & - & * & $\cdots$ & * 
\end{tabular}}
\end{center}

\noindent Here ``\texttt{*}'' indicates a nonzero entry, and ``\texttt{-}'' indicates a zero entry. We have
$$
\begin{array}{l}
\displaystyle \kappa_{p,0}(\pi_{i,j}) = \frac{(r-1)!(i-p)(r-j+1-p)}{(r+1-p)!p!}~~\text{ for $0 \leq p \leq i-1$};\\[3pt]
\displaystyle \kappa_{p,1}(\pi_{i,j}) = \frac{(r-1)!(p+1-i)(r-j-p)}{(r-p)! (p+1)!}~~\text{ for $i \leq p \leq r-j-1$};\\[3pt]
\displaystyle \kappa_{p,2}(\pi_{i,j}) = \frac{(r-1)!(p+2-i)(p-r+j+1)}{(r-p-1)!(p+2)!}~~\text{ for $r-j \leq p \leq r-1$}.
\end{array}
$$

\begin{proposition}\label{prop:logconcavityforcurves}
$(1)$ The Betti table of $R(C, B; L)$ is asymptotically pure:
$$
\frac{a_{i,j}}{d} \rightarrow \begin{cases} 1 & \text{if $i=c$ and $j=c'$} \\ 0 & \text{otherwise} \end{cases}~~\text{ as $d \to \infty$}.
$$
$(2)$ $\kappa_{p,0}(C, B; L), \ldots, \kappa_{c-1, 0}(C, B; L)$ is  increasing, and $\kappa_{r-c', 2}(C, B; L), \ldots, \kappa_{r-1, 2}(C, B; L)$ is decreasing.\\[3pt]
$(3)$ The Betti numbers $\kappa_{p,1}(C, B; L)$ form a unimodal sequence.
\end{proposition}

\begin{proof}
$(1)$ Let $\overline{a}_{i,j}:=a_{i,j}/d$. Then (\ref{eq:bsdecomp}) says
$$
\sum_{i=0}^{c} \sum_{j=0}^{c'} \overline{a}_{i,j} = 1.
$$ 
Notice that $\kappa_{0,0}(\pi_{i,j}) \to i/d$ and $\kappa_{r-1, 2}(\pi_{i,j}) \to j/d$ as $d \to \infty$.
Since $\kappa_{0,0}(C, B; L) = c$ and $\kappa_{r-1, 2}(C, B; L) = c'$, it follows from (\ref{eq:bsdecomp}) that 
$$
\sum_{i=0}^c \sum_{j=0}^{c'} i \overline{a}_{i,j} \to c~~\text{ and }~~\sum_{i=0}^c \sum_{j=0}^{c'} j \overline{a}_{i,j} \to c' ~~\text{ as $d \to \infty$}.
$$
Then we have
$$
\sum_{i=0}^c \sum_{j=0}^{c'} (c-i) \overline{a}_{i,j} \to 0~~\text{ and }~~\sum_{i=0}^c \sum_{j=0}^{c'} (c'-j) \overline{a}_{i,j} \to 0 ~~\text{ as $d \to \infty$}.
$$
Thus $\overline{a}_{i,j} \to 0$ as $d \to \infty$ unless $i=c$ and $j=c'$, and hence, $\overline{a}_{c, c'} \to 1$ as $d \to \infty$.

\medskip

\noindent $(2)$ As $d \gg 0$, we have $r \approx d$ and
$$
\kappa_{p,0}(\pi_{i,j}) \approx \frac{i-p}{p!} d^{p-1} ~~\text{ for $0 \leq p \leq i-1$}.
$$
Then $(1)$ implies that 
$$
\kappa_{p,0}(C, B; L) \approx \frac{c-p}{p!}d^p~~\text{ for $0 \leq p \leq c-1$}.
$$
Thus the first assertion holds, and the second assertion follows from Proposition \ref{prop:duality}.

\medskip

\noindent $(3)$ For $(r-2+c)/2 \leq p \leq r-j-1$,  we find
$$
\frac{\kappa_{p+1, 1}(\pi_{i,j})}{\kappa_{p, 1}(\pi_{i,j})} = \frac{(r-p) (p+2-i)(r-j-p-1)}{(p+2)(p+1-i)(r-j-p)} \leq 1
$$
since $r- p \leq p+2$ and 
$$
(p+2-i)(r-j-p-1)-(p+1-i)(r-j-p) = (r-j-p-1)-(p+1-i)  \leq 0.
$$
For $(r-2+c)/2 \leq p \leq r-2$, we get\\[-20pt]

\begin{small}
$$
\kappa_{p+1,1}(C, B; L)=\sum_{i=0}^{c} \sum_{j=0}^{\min \{c', r-p-2\}} a_{i,j}\kappa_{p+1, 1}(\pi_{i,j}) \leq \sum_{i=0}^c \sum_{j=0}^{\min\{c', r-p-1\} }a_{i,j}\kappa_{p,1}(\pi_{i,j})=\kappa_{p,1}(C, B; L),
$$
\end{small}

\noindent so the Betti numbers $\kappa_{p,q}(C, B; L)$ with $(r-2+c)/2 \leq p \leq r-1$ form a  decreasing sequence. Now, as in \cite[Proof of Proposition A]{EEL1}, we compute
\begin{equation}\label{eq:kappa_{p,1}(C, B; L)}
\kappa_{p,1}(C, B; L)  = \chi(C, \wedge^p M_L \otimes B \otimes L) - {r+1 \choose p+1}c
 = {r \choose p} \left( b-\frac{pd}{r} - \frac{(r+1)c}{p+1} \right)
\end{equation}
for $c-1 \leq p \leq r-c'$, where $b:=r + \deg B + 1$. For $(r-1)/2 \leq p \leq (r-2+c)/2$, we have
$$
\frac{\kappa_{p+1,1}(C, B; L)}{\kappa_{p,1}(C, B; L)} = \frac{(r-p)(b-(p+1)d/r - (r+1)c/(p+2))}{(p+1)(b-pd/r - (r+1)c/(p+1))} \leq 1
$$
since $r-p \leq p+1$ and 
$$
\left(b-\frac{(p+1)d}{r} - \frac{(r+1)c}{p+2} \right) - \left(b-\frac{pd}{r} - \frac{(r+1)c}{p+1} \right)= -\frac{d}{r} + \frac{(r+1)c}{(p+1)(p+2)} \approx -1 + \frac{dc}{d/2 \cdot d/2}< 0.
$$
Thus $\kappa_{p+1, 1}(C, B; L) \leq \kappa_{p,1}(C, B; L)$. We have shown that the Betti numbers $\kappa_{p,q}(C, B; L)$ with $(r-1)/2 \leq p \leq r-1$ form a  decreasing sequence.
By Proposition \ref{prop:duality}, the Betti numbers $\kappa_{p,q}(C, B; L)$ with $0 \leq p \leq (r-1)/2$ form an increasing sequence.
\end{proof}

\begin{remark}
When $B= \sO_C$, Proposition \ref{prop:logconcavityforcurves} $(1)$ is the main theorem of \cite{Erman}. In view of Proposition \ref{prop:logconcavityforcurves} $(2)$, one may expect that nonzero entries of the top row ($q=0$) of the Betti table form an increasing sequence and nonzero entries of the bottom row ($q=n+1$) of the Betti table form a decreasing sequence in higher dimensions.
\end{remark}

\begin{example}
Recall that a log-concave sequence of positive terms is unimodal. It is tempting to expect that $\kappa_{p,1}(C, B; L)$ form a log-concave sequence. Unfortunately, this may fail when $p$ is small. For instance, let $C$ be a general smooth projective complex curve of genus $3$, and $B:=\omega_C(x)$ for a point $x \in C$. Note that $B$ is not base point free, $\deg B = 5$, and $h^0(C, B) = 3$. If $L$ is a very ample line bundle on $C$ of degree $d \gg 0$, then \cite[Theorem C]{EL3} says that $\kappa_{1,1}(C, B; L)$ is a polynomial in $d$ of degree
$$
\gamma_1(B):=\dim \{ \xi \in C_2 \mid \underbrace{\text{$H^0(C, B) \to H^0(\xi, B|_{\xi})$ is not surjective}}_{\Longleftrightarrow ~h^1(C, B(-\xi))=h^0(C, \sO_C(\xi-x))=1~ \Longleftrightarrow ~ x \in \xi} \} = 1
$$
On the other hand, from (\ref{eq:kappa_{p,1}(C, B; L)}), we find
$$
\begin{array}{l}
\kappa_{2,1}(C, B; L) = {d-3 \choose 2} \left( d+3 - \frac{2d}{d-3} - \frac{(d-2)3}{3} \right) \approx \frac{3}{2} d^2;\\[3pt]
\kappa_{3,1}(C, B; L) = {d-3 \choose 3} \left(d+3 - \frac{3d}{d-3} - \frac{(d-2)3}{4} \right) \approx \frac{1}{24} d^4.
\end{array}
$$
Thus $\kappa_{2,1}(C, B; L)^2 < \kappa_{1,1}(C, B; L) \cdot \kappa_{3,1}(C, B; L)$.
\end{example}

\bibliographystyle{ams}

\end{document}